\documentclass[11pt]{amsart}

\usepackage{verbatim}
\usepackage{mathtools}
\usepackage[bookmarks,colorlinks,breaklinks]{hyperref}  
\hypersetup{linkcolor=blue,citecolor=blue,filecolor=blue,urlcolor=blue}

\newtheorem{theorem}{Theorem}[section]
\newtheorem{proposition}[theorem]{Proposition}
\newtheorem{lemma}[theorem]{Lemma}

\newtheorem{definition}[theorem]{Definition}

\theoremstyle{plain}
\numberwithin{equation}{theorem}

\theoremstyle{remark}

\newcommand{\F}{{\mathbb F}}

\newcommand{\Kbar}{\overline{K}}

\newcommand{\tensor}{\otimes}

\DeclareMathOperator{\Frac}{Frac}

\DeclareMathOperator{\sep}{sep}
\DeclareMathOperator{\tor}{tor}

\newcommand{\bG}{{\mathbb G}}
\newcommand{\bN}{{\mathbb N}}

\newcommand{\bF}{{\mathbb F}}
\newcommand{\Fq}{\bF_q}

\newcommand{\hhat}{{\widehat h}}

\newcommand{\Drin}{\Phi}

\title{A variant of Siegel's theorem for Drinfeld modules}
\author{Simone Coccia and Dragos Ghioca}


\address{
Simone Coccia \\
Department of Mathematics\\
University of British Columbia\\
Vancouver, BC V6T 1Z2\\
Canada
}

\email{scoccia@math.ubc.ca}

\address{
Dragos Ghioca \\
Department of Mathematics\\
University of British Columbia\\
Vancouver, BC V6T 1Z2\\
Canada
}

\email{dghioca@math.ubc.ca}

\begin{document}

\begin{abstract}
  We complete the proof of a Siegel type statement for finitely generated $\Phi$-submodules of $\mathbb{G}_a$ under the action of a Drinfeld module $\Phi$.
\end{abstract}

\maketitle

\section{Introduction}\label{intro}

In 1929, Siegel (\cite{Siegel}) proved that if $C$ is an irreducible
affine curve defined over a number field $K$ and $C$ has at least
three points at infinity, then there are at most finitely many
$K$-rational points on $C$ that have integral coordinates. The two most important ingredients in the proof
of Siegel's theorem are Diophantine approximation, along with the
fact that certain groups of rational points are finitely generated;
when $C$ has genus greater than $0$, the group in question is the
Mordell-Weil group of the Jacobian of $C$, while when $C$ has genus $0$,
the group in question is the group of $S$-units in a finite extension
of $K$.  Motivated by the analogy between rank $2$ Drinfeld modules and
elliptic curves (along with the understanding that higher rank Drinfeld modules of generic characteristic are the right vehicle in characteristic $p$ for similar conjectures to Diophantine questions one would pose for abelian varieties in characteristic $0$), Ghioca and Tucker conjectured in \cite[Conjecture~5.5]{findrin} a Siegel
type statement for finitely generated $\Phi$-submodules $\Gamma$ of
$\bG_a(K)$ (where $q$ is a power of $p$, $\Phi$ is a generic characteristic Drinfeld module of arbitrary rank and $K$ is a finite extension of $\F_q(t)$). In \cite{Siegel-Drinfeld}, Ghioca and Tucker proved the Siegel type statement for Drinfeld modules under the technical hypothesis that the ground field $K$ admits a single place which lies over the place at infinity of $\F_q(t)$. In the current paper, we are able to remove this technical hypothesis on the field $K$ and prove the following general result for Drinfeld modules in the spirit of the famous Siegel's theorem. 

\begin{theorem}
\label{Siegel}
Let $q$ be a power of the prime number $p$, let $K$ be a finite extension of the function field $\Fq(t)$ and let  $\Phi$ be a Drinfeld module of generic characteristic defined over $K$. Let $\Gamma$ be a finitely generated $\Phi$-submodule of $\mathbb{G}_a(K)$, let $\alpha\in K$, and let $S$ be a finite set of places of the field $K$. Then there are finitely many $\gamma\in\Gamma$ such that $\gamma$ is $S$-integral with respect to $\alpha$.
\end{theorem}

As an aside, we note that for any finitely generated $\Phi$-module $\Gamma\subset \Kbar$, we can always find a finite extension of $K$ containing $\Gamma$ and also containing the point $\alpha$; this explains why we state our Theorem~\ref{Siegel} assuming $\Gamma$, $\alpha$ and also the coefficients of $\Phi_t$ are contained in $K$ (rather than contained in $\Kbar$, as \cite[Conjecture~5.5]{findrin} was previously stated). 
We refer the reader to Section~\ref{notation} for more details on Drinfeld modules, including the notion of $S$-integral points with respect to a given point.

The strategy of our proof is identical with the strategy employed in \cite{Siegel-Drinfeld}; actually, with one exception, the statements proven in \cite{Siegel-Drinfeld} are valid in the generality of our Theorem~\ref{Siegel}. However, the main result from the aforementioned paper, i.e., \cite[Proposition~3.12]{Siegel-Drinfeld} is proven under the technical assumption that there exists a single place in the function field $K$ lying above the place at infinity from $\Fq(t)$. Moreover, the strategy of proof from \cite[Proposition~3.12]{Siegel-Drinfeld} does not extend to arbitrary function fields $K$, as explained in \cite[Remark~3.14]{Siegel-Drinfeld}. So, the main result of our current paper is to develop an alternative strategy for proving \cite[Proposition~3.12]{Siegel-Drinfeld} for arbitrary function fields $K$; the statement generalizing \cite[Proposition~3.12]{Siegel-Drinfeld} is proven in our Proposition~\ref{prop:main}. Finally, we note that in Theorem~\ref{Siegel}, if $\Phi$ were a Drinfeld module of special characteristic, then its conclusion could fail (as can be easily seen in the case the Drinfeld module is simply given by $\Phi_t(x):=x^q$; then we may take $K$ be a finite extension of the rational function field $\Fq(\theta)$ and so, there are infinitely many points in the cyclic $\Phi$-module generated by $\theta+1$, which are $S$-integral with respect to $1$, where $S$ consists of the places $v$ of $K$ for which $\theta$ is not a $v$-adic unit). 

Our Theorem~\ref{Siegel} completes the proof of Siegel's theorem in the context of Drinfeld modules. Over the past 30 years, there was a significant increase in the study of the arithmetic of Drinfeld modules which established the validity of several classical theorems from the arithmetic geometry of abelian varieties in the context of Drinfeld modules. Indeed, Scanlon \cite{Scanlon} proved the Manin-Mumford type theorem for Drinfeld modules, conjectured by Denis \cite{Denis-conjectures}. The second author proved in \cite{Mat.Ann} an equidistribution statement for torsion points of Drinfeld modules, which may be interpreted as a weaker variant of the classical equidistribution theorem of Szpiro-Ullmo-Zhang \cite{suz} for torsion points of abelian varieties. Breuer \cite{Breuer} proved an Andr\'e-Oort type theorem for Drinfeld modules in the spirit of the classical results of Edixhoven-Yafaev \cite{Edixhoven} for Shimura varieties. The second author (both in a single author paper \cite{IMRN} and also in a joint paper with Tucker \cite{dynml}) proved various instances of a Mordell-Lang type statement for Drinfeld modules conjectured by Denis \cite{Denis-conjectures}. Actually, the paper \cite{dynml} constituted the starting point of the Dynamical Mordell-Lang Conjecture (formulated in \cite{GT-DML}), which by itself generated extensive research in the past 15 years (for a comprehensive discussion of the Dynamical Mordell-Lang Conjecture, see \cite{DML-book}, especially \cite[Chapter~12]{DML-book} which details the connection between the aforementioned conjecture and Denis' conjecture for Drinfeld modules from \cite{Denis-conjectures}).

The plan of our paper is as follows: in Section~\ref{notation} we give the basic definitions and notation, and
then, in Section~\ref{proof} we prove the
main result, which is Proposition~\ref{prop:main} (note that  Theorem~\ref{Siegel} is an immediate consequence of our Proposition~\ref{prop:main} using the exact same strategy of proof as for \cite[Theorem~2.4]{Siegel-Drinfeld}, the only difference being the replacement of \cite[Proposition~3.12]{Siegel-Drinfeld} by our new Proposition~\ref{prop:main}).

{\bf Acknowledgments.} We thank the referee for their many useful suggestions.


\section{Notation}
\label{notation}

Our notational section has a significant overlap with \cite[Section~2]{Siegel-Drinfeld}.


\subsection{Drinfeld modules}
We begin by defining a Drinfeld module.  Let $p$ be a prime and let
$q$ be a power of $p$. Let $A:=\mathbb{F}_q[t]$, let $K$ be a finite field extension of $\mathbb{F}_q(t)$, and let $\Kbar$ be an
algebraic closure of $K$. We let $\tau$ be the Frobenius on
$\mathbb{F}_q$, and we extend its action on $\Kbar$.  Let $K\{\tau\}$
be the ring of polynomials in $\tau$ with coefficients from $K$ (the
addition is the usual addition, while the multiplication is the
composition of functions).

A Drinfeld module is a morphism $\Drin:A\rightarrow K\{\tau\}$ for
which the coefficient of $\tau^0$ in $\Drin(a)=:\Drin_a$ is $a$ for
every $a\in A$, and there exists $a\in A$ such that $\Drin_a\ne
a\tau^0$. The definition given here represents what Goss \cite{Goss}
calls a Drinfeld module of ``generic characteristic''.

We note that usually, in the definition of a Drinfeld module, $A$ is
the ring of functions defined on a projective nonsingular curve $C$,
regular away from a closed point $\eta\in C$. For our definition of a
Drinfeld module, $C=\mathbb{P}^1_{\mathbb{F}_q}$ and $\eta$ is the
usual point at infinity on $\mathbb{P}^1$. On the other hand, every
ring of regular functions $A$ as above contains $\mathbb{F}_q[t]$ as a
subring, where $t$ is a nonconstant function in $A$. Furthermore, even for such a general ring of regular functions $A$, we have that $A$ is a finitely generated module over the ring $\Fq[t]$, which means that the statement of our Theorem~\ref{Siegel} is left unchanged.
 
For every field extension $K\subset L$, the Drinfeld module $\Drin$
induces an action on $\mathbb{G}_a(L)$ by $a*x:=\Drin_a(x)$, for each
$a\in A$. We call \emph{$\Phi$-submodules} subgroups of
$\mathbb{G}_a(\overline{K})$ which are invariant under the action of
$\Phi$. We define the \emph{rank} of a $\Phi$-submodule $\Gamma$ to  be 
$$\dim_{\Frac(A)}\Gamma\tensor_A\Frac(A).$$

A point $\alpha\in\Kbar$ is \emph{torsion} for the Drinfeld module action if
and only if there exists $Q\in A\setminus\{0\}$ such that
$\Drin_Q(\alpha)=0$. It is immediate to see that the set of all torsion points is also a $\Phi$-submodule. Since each polynomial $\Phi_Q$ is separable, the torsion submodule
$\Phi_{\tor}$ lies in the separable closure $K^{\sep}$ of $K$.

\subsection{The coefficients of the Drinfeld module}
Each Drinfeld module is isomorphic over $\Kbar$ to a Drinfeld module for which the leading
 coefficient of $\Drin_t$ equals $1$; in particular, for each $a\in\Fq[t]\setminus \Fq$, we would then have that $\Phi_a$ is a polynomial whose leading coefficient lives in $\Fq$. This is a standard observation used previously in \cite{mw2, locleh}.

\subsection{Valuations and Weil heights}

We let $M_K$ be the set of valuations on $K$. Then $M_K$ is a set of valuations which satisfies a
product formula (see \cite[Chapter 2]{Serre-Mordell_Weil}); thus (with an appropriate normalization for each absolute value $|\cdot |_v$) we have 
\begin{itemize}
\item for each nonzero $x\in K$, there are finitely many $v\in M_K$
such that $|x|_v\ne 1$; and \\
\item for each nonzero $x\in K$, we have $\prod_{v\in M_K} |x|_v=1$.
\end{itemize}
We may use these valuations to define a
Weil height for each $x \in K$ as
\begin{equation}\label{weil}
h(x) = \sum_{v \in M_K} \log \max (|x|_v,1).
\end{equation}

%
%

\subsection{Canonical heights}
Let $\Drin:A\rightarrow K\{\tau\}$ be a Drinfeld module of \emph{rank}
$d$ (i.e. the degree of $\Phi_t$ as a polynomial in $\tau$ equals
$d$). The canonical height of $\beta\in K$ relative to $\Drin$ (see \cite{Denis}) is
defined as
$$\hhat(\beta) = \lim_{n \to \infty}
\frac{h(\Drin_{t^n}(\beta))}{q^{nd}}.$$ 
Denis \cite{Denis} showed that a point is torsion if and only if its canonical height equals $0$. 

For every $v\in M_K$, we let the local canonical height of $\beta\in K$ at $v$ to  be
\begin{equation}\label{poon-def}
\hhat_v(\beta) = \lim_{n \to \infty} \frac{\log
  \max(|\Drin_{t^n}(\beta)|_v, 1)}{q^{nd}};
\end{equation}
for more details on local canonical heights, see \cite{locleh}. 
Furthermore, for every $a\in \Fq[t]$, we have $\hhat_v(\Phi_a(x))=\deg(\Phi_a)\cdot\hhat_v(x)$. 
It is clear that $\hhat_v$ satisfies the triangle inequality, and also that $\sum_{v\in M_K} \hhat_v(\beta) = \hhat(\beta)$ (therefore, also $\hhat(\cdot)$ satisfies the triangle inequality).

\subsection{Integrality and reduction}
Since we can always replace $K$ by a finite extension, we will define the notion of $S$-integrality with respect to a given point $\alpha$ under the assumption that $\alpha\in K$ (note that we are interested in Theorem~\ref{Siegel} for the $S$-integrality with respect to $\alpha$ within a given finitely generated $\Phi$-module $\Gamma$ and therefore, once again at the expense of replacing $K$ by a finite extension, we may assume $\Gamma\subset K$). 
\begin{definition}
\label{S-integral definition}
For a finite set of places $S \subset M_K$ and $\alpha\in K$, we say
that $\beta \in K$ is $S$-integral with respect to $\alpha$ if for
every place $v\notin S$, the following
are true:
\begin{itemize}
\item if $|\alpha|_v\le 1$, then $|\alpha-\beta|_v\ge 1$.
\item if $|\alpha|_v >1$, then $|\beta|_v\le 1$.
\end{itemize}
\end{definition}
We note that if $\beta$ is $S$-integral with respect to $\alpha$, then it is also $S'$-integral with respect to $\alpha$, where $S'$ is a finite set of places containing $S$. Furthermore (as noted in \cite[Subsection~2.6]{Siegel-Drinfeld}), the notion of $S$-integrality with respect to a point is invariant when replacing $K$ by a finite extension $L$ and also replacing $S$ by the set of places in $L$ lying above the places from $S$.

\begin{definition}
  The Drinfeld module $\Phi$ has good reduction at a place $v$ if for
  each nonzero $a\in A$, all coefficients of $\Phi_a$ are $v$-adic
  integers and the leading coefficient of $\Phi_a$ is a $v$-adic unit.
  If $\Phi$ does not have good reduction at $v$, then we say that
  $\Phi$ has bad reduction at $v$.
\end{definition}

It is immediate to see that $\Phi$ has good reduction at $v$ if and only if all coefficients of $\Phi_t$ are $v$-adic integers, since we already assumed that the leading coefficient of $\Phi_t$ equals $1$. Finally, we note that replacing $\Phi$ by the isomorphic Drinfeld module $\Psi$ given by $\Psi_t(x)=c\Phi_t\left(c^{-1}x\right)$ (for some nonzero $c\in K$), if we were to show that there exist finitely many $S$-integral points in the finitely generated $\Phi$-module $\Gamma$ with respect to a given point $\alpha$, then it suffices to prove that there exist finitely many points in the $\Psi$-module $c\cdot \Gamma$, which are $S'$-integral with respect to $c\alpha$, where $S'$ consists of all places in $S$ along with the finitely many places $v$ of $K$ for which $|c|_v\ne 1$. Indeed, for each element $c\gamma\in c\Gamma$ which is $S'$-integral with respect to $c\alpha$, we must have that $\gamma$ is $S'$-integral with respect to $\alpha$. So, knowing there exist at most finitely many such elements $\gamma$ which are $S'$-integral for $\alpha$ yields that there exist at most finitely many points in $\Gamma$, which are $S$-integral with respect to $\alpha$ (because $S\subseteq S'$ and enlarging the set $S$ may only increase the number of $S$-integral points with respect to a given point).


\section{Proofs of our main results}
\label{proof}

Before proceeding to the proof of Theorem~\ref{Siegel}, we prove several facts about local heights. 

%
%

From now on, let $\Phi_t=\sum_{i=0}^d a_i\tau^i$; also, as explained in Section~\ref{notation}, we assume from now on that $a_d=1$.  The following result follows using a similar proof as the one employed in  \cite[Fact~3.6]{Siegel-Drinfeld} (in the case $\Phi_t$ is monic).

\begin{lemma}
\label{jumps}
For every place $v$ of $K$, there exists a real number $M_v\ge 1$ such that for each $x\in K$, if $|x|_v> M_v$, then for every nonzero $Q\in \Fq[t]$, we have $|\Phi_{Q}(x)|_v=|x|_v^{q^{d\cdot\deg(Q)}}$. Moreover, if $|x|_v> M_v$, then $\hhat_v(x)=\log |x|_v$.
\end{lemma}

\begin{proof}
We let $M_v=\max\left\{1,\max_{i=0}^{d-1}|a_i|_v^{1/(q^d-q^i)}\right\}$. Then for each $x\in K$ such that $|x|_v>M_v$, we have that $|\Phi_t(x)|_v=|x|_v^{q^d}$; so, in particular, $|\Phi_t(x)|_v>M_v$ and therefore, for each nonnegative integer $n$, we have that $\left|\Phi_{t^n}(x)\right|_v=|x|_v^{q^{dn}}$. In particular, for any nonzero polynomial $Q\in A$, we have $|\Phi_Q(x)|_v=|x|_v^{q^{d\deg(Q)}}$. Finally, this means that if $|x|_v>M_v$, then $\hhat_v(x)=\log|x|_v$. 
\end{proof}

\begin{lemma}
\label{lem:local canonical}
For each place $v$ of $M_K$ and for each $x\in K$, if $\hhat_v(x)>0$, then for all polynomials $Q\in\Fq[t]$ of sufficiently large degree, we have that $\hhat_v(x)=\frac{\log\left|\Phi_Q(x)\right|_v}{q^{d\cdot \deg(Q)}}$. 
\end{lemma}

\begin{proof}
The proof is an immediate corollary of Lemma~\ref{jumps} once we note that if $\hhat_v(x)>0$ then there exists a nonzero integer $\ell$ (depending on $x$ and $v$) such that $|\Phi_{t^\ell}(x)|_v>M_v$; moreover, we may assume $\ell$ is minimal with this property. Then Lemma~\ref{jumps} yields that 
\begin{equation}
\label{eq:1st large iterate}
\hhat_v(x)=\frac{\hhat_v\left(\Phi_{t^\ell}(x)\right)}{q^{d\ell}}= \frac{\log|\Phi_{t^\ell}(x)|_v}{q^{d\ell}}.
\end{equation}
Moreover, for each polynomial $Q\in\Fq[t]$ of degree at least equal to $\ell$, we have that 
\begin{equation}
\label{eq:2nd large iterate}
\left|\Phi_Q(x)\right|_v=\left|\Phi_{t^\ell}(x)\right|_v^{q^{d(\deg(Q)-\ell)}}. 
\end{equation}
Equations \eqref{eq:1st large iterate} and \eqref{eq:2nd large iterate} finish the proof of Lemma~\ref{lem:local canonical}.
\end{proof}

The following result is an immediate consequence of \cite[Theorem~4.5]{mw2} (which provides a more general positive lower bound for the canonical height of non-torsion points $x\in\Kbar$ depending only on the number of places of bad reduction for the given Drinfeld module in the field extension $K(x)$).

\begin{lemma}
\label{lem:lower bound canonical}
There exists a positive constant $c_0$ such that for all non-torsion points $x\in K$, we have $\hhat(x)\ge c_0$.
\end{lemma}

The next result will be used in the proof of Proposition~\ref{prop:main} (which is our main technical ingredient for the proof of Theorem~\ref{Siegel}).

\begin{lemma}
\label{lem:local canonical leading term}
Let $r\in\bN$, let $\gamma_1,\dots, \gamma_r\in K$ and let $v\in M_K$. Assume that 
\begin{equation}
\label{eq:hypothesis local canonical 1}
\hhat_v(\gamma_1)>\frac{\max_{i>1}\hhat_v(\gamma_i)}{q^d}.
\end{equation}
Then there exists $n_0\in\bN$ (depending only on $\Phi$, $v$ and on $\gamma_1,\dots, \gamma_r$) and there exists a positive real number $c_1$ (depending only on $\Phi$ and $v$)  such that for all polynomials $P_1,\dots, P_r\in\Fq[t]$ satisfying
\begin{enumerate}
\item[(i)] $\deg(P_1)>\max\{\deg(P_2),\cdots,\deg(P_r)\}$, and
\item[(ii)] $\min\{\deg(P_2),\cdots,\deg(P_r)\}\ge n_0$,
\end{enumerate}
we have $\log\left|\Phi_{P_1}(\gamma_1)+\cdots +\Phi_{P_r}(\gamma_r)\right|_v> c_1\cdot q^{d\cdot (\deg(P_1)-n_0)}$.
\end{lemma}

\begin{proof}
Let $M_v\ge 1$ be the real number as in the conclusion of Lemma~\ref{jumps} and let $L_v:=2M_v$. Then for each $i=2,\dots, r$ such that $\hhat_v(\gamma_i)=0$, we must have that for all polynomials $Q_i\in\Fq[t]$ then 
\begin{equation}
\label{eq:3rd local canonical 0}
|\Phi_{Q_i}(\gamma_i)|_v< L_v. 
\end{equation}
Similarly, according to Lemma~\ref{lem:local canonical}, we have that for each $i=1,\dots, r$ for which $\hhat_v(\gamma_i)>0$ (note that our hypothesis yields that $\hhat_v(\gamma_1)>0$) then for all polynomials $Q_i\in\Fq[t]$ of degree larger than some positive integer $n_0$ (which depends only on $v$ and on the $\gamma_i$'s),
\begin{equation}
\label{eq:3rd local canonical}
\log\left|\Phi_{Q_i}(\gamma_i)\right|_v=q^{d\deg(Q_i)}\cdot \hhat_v(\gamma_i).
\end{equation} 
Furthermore, for each $i$ as in equation \eqref{eq:3rd local canonical}, as long as $\deg(Q_i)$ is sufficiently large, then we have that 
\begin{equation}
\label{eq:3rd local canonical 1}
|\Phi_{Q_i}(\gamma_i)|_v\ge L_v>M_v.
\end{equation}
Combining equations \eqref{eq:3rd local canonical 0}, \eqref{eq:3rd local canonical} and \eqref{eq:3rd local canonical 1}, coupled with our hypotheses \eqref{eq:hypothesis local canonical 1} and~(i) from Lemma~\ref{lem:local canonical leading term}, we obtain that $\left|\Phi_{P_1}(\gamma_1)\right|_v>\max_{i>1}\left|\Phi_{P_i}(\gamma_i)\right|_v$ and therefore,  
\begin{equation}
\label{eq:4th local canonical}
\left|\sum_{i=1}^r\Phi_{P_i}(\gamma_i)\right|_v=\left|\Phi_{P_1}(\gamma_1)\right|_v > M_v.  
\end{equation}
Finally, combining \eqref{eq:4th local canonical} with Lemma~\ref{jumps} (see also equation \eqref{eq:3rd local canonical}), we obtain that 
$$\log\left|\sum_{i=1}^r\Phi_{P_i}(\gamma_i)\right|_v= \log\left|\Phi_{P_1}(\gamma_1)\right|_v = 
 q^{d\cdot (\deg(P_1)-n_0)}\cdot \log\left|\Phi_{t^{n_0}}(\gamma_1)\right|_v.$$
Since $|\Phi_{t^{n_0}}(\gamma_1)|_v\ge L_v>M_v$ and therefore, letting $c_1:=\log(2L_v/3)\ge \log(4/3)>0$, we obtain the desired conclusion in Lemma~\ref{lem:local canonical leading term}.
\end{proof}

The following proposition is the key technical result required to
prove Theorem \ref{Siegel} and it is the generalization of \cite[Proposition~3.12]{Siegel-Drinfeld} in the case of arbitrary function fields $K$.

\begin{proposition}
\label{prop:main}
 Let $\alpha\in K$, let $\Gamma$ be a torsion-free $\Phi$-submodule of $\bG_a(K)$, and let $\gamma_1, \dots, \gamma_r$ be a basis for the $\Phi$-submodule $\Gamma$. For each $i\in\{1,\dots,r\}$ let $(P_{n,i})_{n\in\mathbb{N}}\subset\Fq[t]$ be a sequence of polynomials such that for each $m\ne n$, the $r$-tuples $(P_{n,i})_{1\le i\le r}$ and $(P_{m,i})_{1\le i\le r}$ are distinct.  Then there exists a place $v\in M_K$
  such that
\begin{equation}
\label{positive}
\limsup_{n\to\infty} \frac{ \log |\sum_{i=1}^r
\Phi_{P_{n,i}}(\gamma_i) - \alpha |_v}{\sum_{i=1}^r q^{d\deg
P_{n,i}}}> 0.
\end{equation}
\end{proposition}

\begin{proof}
The hypothesis on the $r$-tuples $(P_{n,i})_{1\le i\le r}$ implies that
\begin{equation*}
	\lim_{n \to \infty} \sum_{i=1}^r q^{d\deg P_{n,i}}=+\infty.
\end{equation*}
Combining this with the triangle inequality for the $v$-adic norm, we have that the sought statement is equivalent to the existence of a place $v$ such that
\begin{equation}
\label{positive simplified}
\limsup_{n\to\infty} \frac{ \log |\sum_{i=1}^r
	\Phi_{P_{n,i}}(\gamma_i)|_v}{\sum_{i=1}^r q^{d\deg
		P_{n,i}}}> 0.
\end{equation}
Observe that it is sufficient to prove it for a subsequence $(n_k)_{k\ge 1} \subset \bN$, since passing to a subsequence can only lower the $\limsup$. We will repeatedly use this fact during the proof and we will drop the extra indexes in order to lighten the notation.
	
The proof proceeds by induction on $r$. If $r=1$, since $\gamma_1$ is non-torsion, there exists a place $v$ such that $\hhat_v(\gamma_1)>0$ and the conclusion follows from Lemma~\ref{lem:local canonical}. We can then assume that \eqref{positive simplified} holds true for all $\Phi$-submodules of rank less than $r$ and we are going to prove it for all $\Phi$-submodules of rank $r$.

Let $S_0$ be the set of places $v \in M_K$ such that $\hhat_v(\gamma)>0$ for some $\gamma \in \Gamma$. This set is finite, as proved in \cite[Fact~3.13]{Siegel-Drinfeld}. Furthermore, for any place $w$ outside the set $S_0$, we have that $|\gamma|_w$ is uniformly bounded above for all $\gamma\in \Gamma$ (see Lemma~\ref{jumps}).

If there exists $j$ such that
\begin{equation}
\label{lower rank reduction}
	\lim_{n \to \infty} \frac{q^{d\deg P_{n,j}}}{\sum_{i=1}^r q^{d\deg P_{n,i}}}=0
\end{equation}
then the conclusion follows from the inductive hypothesis. Indeed, by the inductive hypothesis, there exists a place $v$  such that
\begin{equation}
\label{exp}
	\limsup_{n\to\infty} \frac{ \log |\sum_{i\neq j} \Phi_{P_{n,i}}(\gamma_i)|_v}{\sum_{i\neq j} q^{d\deg P_{n,i}}}> 0.
\end{equation}
So, in particular, there exists some $i\ne j$ such that $|\Phi_{P_{n,i}}(\gamma_i)|_v\to\infty$ as $n\to\infty$ (also note that the denominator from \eqref{exp} must go to infinity due to equation \eqref{lower rank reduction}). So, $\hhat_v(\gamma_i)>0$ (by Lemma~\ref{jumps}), which means that  actually $v\in S_0$. Combining \eqref{exp} with \eqref{lower rank reduction} gives
\begin{equation}
\label{positive no j}
\limsup_{n\to\infty} \frac{ \log |\sum_{i\neq j} \Phi_{P_{n,i}}(\gamma_i)|_v}{\sum_{i=1}^r q^{d\deg P_{n,i}}}> 0.
\end{equation}
We distinguish two cases. If $\hhat_v(\gamma_j)=0$, then $\{|\Phi_{Q}(\gamma_j)|_v\}_{Q \in \bF_q[t]}$ is bounded, so that for large enough $n$ we have
\begin{equation*}
	\left| \sum_{i\neq j} \Phi_{P_{n,i}}(\gamma_i)  \right|_v=\left|\sum_{i=1}^r \Phi_{P_{n,i}}(\gamma_i)  \right|_v
\end{equation*}
and the result follows from \eqref{positive no j}.

If $\hhat_v(\gamma_j)>0$ but $\deg(P_{n,j})\to\infty$ as $n\to\infty$, we apply Lemma~\ref{lem:local canonical} to get
\begin{equation*}
	\lim_{n \to \infty} \frac{ \log |\Phi_{P_{n,j}}(\gamma_j)|_v}{\sum_{i=1}^r q^{d\deg P_{n,i}}}=\hhat_v(\gamma_j) \cdot  \lim_{n \to \infty} \frac{q^{d\deg P_{n,j}}}{\sum_{i=1}^r q^{d\deg P_{n,i}}}=0.
\end{equation*}
Now, if $\deg(P_{n,j})$ is bounded above, then also $\left|\Phi_{P_{n,j}}(\gamma_j)\right|_v$ is bounded above and therefore, once again
\begin{equation*}
\lim_{n \to \infty} \frac{ \log |\Phi_{P_{n,j}}(\gamma_j)|_v}{\sum_{i=1}^r q^{d\deg P_{n,i}}}=0.
\end{equation*}
Thus, combining the fact that $\lim_{n \to \infty} \frac{ \log |\Phi_{P_{n,j}}(\gamma_j)|_v}{\sum_{i=1}^r q^{d\deg P_{n,i}}}=0$ with equation \eqref{positive no j}, gives
\begin{equation}
\label{max min derivation last eq}
\left| \sum_{i\neq j} \Phi_{P_{n,i}}(\gamma_i)  \right|_v=\left|\sum_{i=1}^r \Phi_{P_{n,i}}(\gamma_i)  \right|_v
\end{equation}
for large enough $n$. We then derive \eqref{positive simplified} using \eqref{lower rank reduction} and \eqref{positive no j}.

We therefore assume that there is no $j$ for which \eqref{positive no j} holds. Equivalently, there exists $B \ge 1$ such that for any $n$ we have
\begin{equation*}
	\frac{\max_{1\le i \le r}q^{d\deg P_{n,i}}}{\min_{1\le i \le r}q^{d\deg P_{n,i}}} \le B,
\end{equation*}
which is the same as
\begin{equation}
\label{max min difference}
	\max_{1\le i \le r}\deg P_{n,i}-\min_{1\le i \le r}\deg P_{n,i}\le \frac{\log_qB}{d}.
\end{equation}

We will proceed doing analysis at the places $v\in S_0$; note that $S_0$ is non-empty since each $\gamma_i$ is non-torsion (and so, there are places $v$ such that $\hhat_v(\gamma_i)>0$). 

The strategy of the proof goes as follows: if we cannot prove directly \eqref{positive simplified}, then we find $\delta_1,\dots,\delta_r \in \Gamma$, a sequence $(n_k)_{k\ge 1}\subset \bN$ and a sequence of $r$-tuples of polynomials $(R_{k,i})_{1\le i\le r}$ such that
\begin{equation}
\label{new sequence equality 2}
	\sum_{i=1}^r \Phi_{P_{n_k,i}}(\gamma_i)=\sum_{i=1}^r \Phi_{R_{k,i}}(\delta_i) 
\end{equation}
and
\begin{equation*}
	0<\liminf_{k\to\infty} \frac{\sum_{i=1}^r q^{d\deg P_{n_k,i}}}{\sum_{i=1}^r q^{d\deg R_{k,i}}}\le \limsup_{k\to\infty} \frac{\sum_{i=1}^r q^{d\deg P_{n_k,i}}}{\sum_{i=1}^r q^{d\deg R_{k,i}}}< +\infty.
\end{equation*}
Thanks to these two relations, we have that if there exists a place $v$ for which
\begin{equation*}
\limsup_{k\to\infty} \frac{ \log |\sum_{i=1}^r
	\Phi_{R_{k,i}}(\delta_i)|_v}{\sum_{i=1}^r q^{d\deg
		R_{k,i}}}> 0,
\end{equation*}
then also
\begin{equation*}
\limsup_{k\to\infty} \frac{ \log |\sum_{i=1}^r
	\Phi_{P_{n_k,i}}(\gamma_i) |_v}{\sum_{i=1}^r q^{d\deg
		P_{n_k,i}}}> 0.
\end{equation*}
In this way we reduce to proving \eqref{positive simplified} for the $\delta_i$'s and the sequence of $r$-tuples of polynomials $(R_{k,i})_{1\le i\le r}$. We will choose them in such a way that the process cannot go on forever, thereby showing that \eqref{positive simplified} has to hold at a finite step for a suitable place $v$.

We claim that, if there exist $j$ and $v$ such that
\begin{equation*}
	\hhat_v(\gamma_j)>\frac{1}{q^d}\max_{i \neq j} \hhat_v(\gamma_i)
\end{equation*}
and
\begin{equation*}
	\deg P_{n,j}>\max_{i \neq j} \deg P_{n,i}
\end{equation*}
for all $n$ big enough, then inequality \eqref{positive simplified} holds for the place $v$. Indeed, using \eqref{max min difference}, we have that for any $i$ the degrees $\deg P_{n,i}$ go to infinity as $n \to \infty$. This, together with the two inequalities above, allows us to apply Lemma~\ref{lem:local canonical leading term} to obtain (for a suitable positive integer $n_0$) that 
\begin{equation*}
	\limsup_{n\to\infty} \frac{ \log |\sum_{i=1}^r \Phi_{P_{n,i}}(\gamma_i)|_v}{\sum_{i=1}^r q^{d\deg P_{n,i}}}>c_1q^{-dn_0} \cdot \limsup_{n\to\infty} \frac{q^{d\deg P_{n,j}}}{\sum_{i=1}^r q^{d\deg P_{n,i}}}>0,
\end{equation*}
where we used again \eqref{max min difference}.

In particular, if there exists $j$ such that
\begin{equation*}
	\deg P_{n,j}>\max_{i \neq j} \deg P_{n,i}
\end{equation*}
for all $n$ big enough, either \eqref{positive simplified} follows, or we can assume that for \emph{all} places $v$ we have
\begin{equation*}
	\hhat_v(\gamma_j) \le \frac{1}{q^d}\max_{i \neq j} \hhat_v(\gamma_i) \le \frac{1}{q^d}\max_{1\le i \le r} \hhat_v(\gamma_i).
\end{equation*}

We will now describe the construction of the $\delta_i$'s and $R_{n,i}$'s we referred to before. In the first step of our process, for all $n$ and each $i>1$,  we divide (with quotient and remainder) $P_{n,i}$ by $P_{n,1}$, obtaining
\begin{equation*}
	P_{n,i}=P_{n,1}\cdot C_{n,i}+R_{n,i},
\end{equation*}
so that $\deg R_{n,i}<\deg P_{n,1}$; also, we note that $\deg R_{n,i}\le \deg P_{n,i}$. We also let $R_{n,1}:=P_{n,1}$.  From \eqref{max min difference} the degrees $\deg C_{n,i}$ are uniformly bounded as $n \to \infty$. It follows that there are only finitely many possible polynomials $C_{n,i}$, so that, passing to a subsequence $(n_k)_{k\ge 1}$, we may assume that there exist polynomials $C_i$ satisfying
\begin{equation*}
	C_{n,i}=C_i
\end{equation*}
for all $n$ (where we dropped the index $k$ of the subsequence).

Let
\begin{equation*}
\delta_i =
\begin{cases*}
\gamma_1 + \sum_{j=2}^r \Phi_{C_j}(\gamma_j) & if $i=1$ \\
\gamma_i       & otherwise.
\end{cases*}
\end{equation*}

Observe that for each $n$, we have
\begin{equation}
\label{new sequence equality}
\sum_{i=1}^r \Phi_{P_{n,i}}(\gamma_i)=\sum_{i=1}^r \Phi_{R_{n,i}}(\delta_i),
\end{equation}
as it follows from the definition of the $\delta_i$'s and of the $R_{n,i}$'s.

Also, since $\deg R_{n,i} \le \deg P_{n,i}$ for all $n$ and $i$, we have
\begin{equation*}
	\sum_{i=1}^r q^{d\deg R_{n,i}}\le \sum_{i=1}^r q^{d\deg P_{n,i}},
\end{equation*}
which implies 
\begin{equation*}
	0<\liminf_{n\to\infty} \frac{\sum_{i=1}^r q^{d\deg P_{n,i}}}{\sum_{i=1}^r q^{d\deg R_{n,i}}}.
\end{equation*}
Using that $R_{n,1}=P_{n,1}$ we obtain
\begin{align*}
	\frac{\sum_{i=1}^r q^{d\deg P_{n,i}}}{\sum_{i=1}^r q^{d\deg R_{n,i}}}&=\frac{\sum_{i=1}^r q^{d\deg P_{n,i}}}{q^{d\deg P_{n,1}}}\cdot \frac{q^{d\deg P_{n,1}}}{q^{d\deg P_{n,1}}+\sum_{i=2}^r q^{d\deg R_{n,i}}}\\
	&\le \frac{\sum_{i=1}^r q^{d\deg P_{n,i}}}{q^{d\deg P_{n,1}}}\\
	&\le \frac{r q^{d\max_i{\deg P_{n,i}}}}{q^{d\deg P_{n,1}}}\\
	&\le rB,
\end{align*}
where we used \eqref{max min difference} in the last step. In particular we derive
\begin{equation}
\label{new sequence inequality}
0<\liminf_{n\to\infty} \frac{\sum_{i=1}^r q^{d\deg P_{n,i}}}{\sum_{i=1}^r q^{d\deg R_{n,i}}}\le \limsup_{n\to\infty} \frac{\sum_{i=1}^r q^{d\deg P_{n,i}}}{\sum_{i=1}^r q^{d\deg R_{n,i}}}< +\infty.
\end{equation}
As explained before, thanks to \eqref{new sequence equality} and \eqref{new sequence inequality}, we can reduce the problem to the study of the $\delta_i$'s and the sequence of $r$-tuples of polynomials $(R_{n,i})_{1\le i\le r}$.

In order for our strategy to work, we will also need that \eqref{max min difference} has to hold for the polynomials $R_{n,i}$, that is, we want
\begin{equation*}
	\max_{1\le i \le r}\deg R_{n,i}-\min_{1\le i \le r}\deg R_{n,i}\le \frac{\log_qB'}{d}.
\end{equation*}
for a suitable constant $B'$. However, arguing as we did from \eqref{lower rank reduction} to \eqref{max min derivation last eq}, if the above is not satisfied, then there exists a place $w$ such that
\begin{equation*}
\limsup_{n\to\infty} \frac{ \log |\sum_{i=1}^r
	\Phi_{R_{n,i}}(\delta_i)|_w}{\sum_{i=1}^r q^{d\deg
		R_{n,i}}}> 0.
\end{equation*}
Using \eqref{new sequence equality} and \eqref{new sequence inequality} we obtain
\begin{equation*}
\limsup_{n\to\infty} \frac{ \log |\sum_{i=1}^r
	\Phi_{P_{n,i}}(\gamma_i) |_w}{\sum_{i=1}^r q^{d\deg
		P_{n,i}}}> 0,
\end{equation*}
thereby proving the statement. We may then assume that \eqref{max min difference} holds also for the $R_{n,i}$'s (possibly with a different constant $B$).

Since
\begin{equation*}
	\deg R_{n,1}>\max_{i>1} \deg R_{n,i}
\end{equation*}
for all $n$, as we previously observed, if there exists a place $v$ such that
\begin{equation*}
	\hhat_v(\delta_1) > \frac{1}{q^d}\max_{i>1} \hhat_v(\delta_i),
\end{equation*}
then \eqref{positive simplified} holds at the place $v$. Therefore, we may assume that for all places $v$ we have
\begin{equation*}
	\hhat_v(\delta_1) \le \frac{1}{q^d}\max_{i>1} \hhat_v(\delta_i),
\end{equation*}
from which we get
\begin{equation}
\label{inductive inequality 1}
\hhat_v(\delta_1) \le \frac{1}{q^d}\max_{1\le i \le r} \hhat_v(\gamma_i),
\end{equation}
where we used that $\delta_i=\gamma_i$ for $i>1$.

We now repeat the above construction of the $\delta_i$'s and $R_{n,i}$'s using $R_{n,2}$ in place of $P_{n,1}$, that is, we will proceed dividing by $R_{n,2}$ each of the $R_{n,i}$'s for $i\ne 2$. We will still denote the new polynomials with the letter $R$ to not cluster the notation. We will then find a subsequence of $\mathbb{N}$, polynomials $R_{n,i}$ and  $\delta_2$ (given by a suitable $\Phi$-linear combination of $\delta_1$ and of $\gamma_i$ for $i \ge 2$) for which both \eqref{new sequence equality} and \eqref{new sequence inequality} hold (where $\delta_i=\gamma_i$ for $i\ge 3$).

As before, either the result follows or we can assume that \eqref{max min difference} holds and for all places $v$ we have
\begin{equation*}
\hhat_v(\delta_2) \le \frac{1}{q^d}\max_{i \neq 2} \hhat_v(\delta_i).
\end{equation*}
From this and \eqref{inductive inequality 1} we get
\begin{equation*}
\hhat_v(\delta_2) \le  \frac{1}{q^d}\max_{1\le i \le r} \hhat_v(\gamma_i)
\end{equation*}
for all places $v$.

Continuing this process for all $i$ up to $r$, either we obtain \eqref{positive simplified}, or we find a subsequence of $\mathbb{N}$, polynomials $R_{n,i}$ satisfying \eqref{max min difference} and $\delta_1,\dots,\delta_r$ (given by $\Phi$-linear combinations of the original $\gamma_i$'s) satisfying \eqref{new sequence equality} and \eqref{new sequence inequality}, and such that
\begin{equation}
\label{inductive inequality 2}
\hhat_v(\delta_j) \le  \frac{1}{q^d}\max_{1\le i \le r} \hhat_v(\gamma_i)
\end{equation}
for all $j$ and all places $v$. By summing over $v \in S_0$, we find
\begin{equation*}
\hhat(\delta_j) \le  \sum_{v \in S_0} \frac{1}{q^d}\max_{1\le i \le r} \hhat_v(\gamma_i) \le \frac{1}{q^d} \sum_{1\le i \le r} \hhat(\gamma_i)
\end{equation*}
for any $j$.

Repeating the above procedure with the $\delta_i$'s in place of the $\gamma_i$'s, either we can find $v$ such that \eqref{positive simplified} holds, or we construct $\epsilon_1,\dots,\epsilon_r$ such that for all $j$ and $v$ we have
\begin{equation*}
\hhat_v(\epsilon_j) \le  \frac{1}{q^d} \max_{1\le i \le r} \hhat_v(\delta_i).
\end{equation*}
This, combined with \eqref{inductive inequality 2}, gives
\begin{equation*}
\hhat_v(\epsilon_j) \le  \frac{1}{q^{2d}} \max_{1\le i \le r} \hhat_v(\gamma_i)
\end{equation*}
for all $j$ and $v$. Summing over $v \in S_0$ we find
\begin{equation*}
\hhat(\epsilon_j) \le \frac{1}{q^{2d}} \sum_{1\le i \le r} \hhat(\gamma_i)
\end{equation*}
for all $j$. More generally, repeating this process $\ell$ times, one would find elements $\alpha_1^{(\ell)},\dots, \alpha_r^{(\ell)} \in \Gamma$ for which
\begin{equation}
\label{elleq}
\hhat\left(\alpha_j^{(\ell)}\right)\le \frac{1}{q^{\ell d}} \sum_{1\le i \le r} \hhat(\gamma_i)
\end{equation}
for each $j=1,\dots, r$. 
However, by Lemma~\ref{lem:lower bound canonical} there exists a positive constant $c_0$ such that $\hhat(x)\ge c_0$ for all non-torsion $x \in K$. Since $\Gamma$ is torsion-free, the process has to stop at some finite step. Indeed,  otherwise we would get that for large enough $\ell$, the right hand side in the inequality \eqref{elleq} is smaller than $c_0$ because the $\gamma_i$'s are given and therefore $\alpha_j^{(\ell)}$ must be a torsion point in $\Gamma$ and so, $\alpha_j^{(\ell)}=0$ for each $j$. But then $\sum_{i=1}^r \Phi_{P_{n,i}}(\gamma_i)=0$ (see \eqref{new sequence equality}, which holds at each step in our process) for all $n$ sufficiently large, which contradicts the fact that the $\gamma_i$'s are a basis for the $\Phi$-submodule $\Gamma$ (and the $r$-tuples $(P_{n,i})_{i=1}^r$ are distinct). Therefore, there exists a place $v$ for which \eqref{positive simplified} must hold.
\end{proof}

Theorem~\ref{Siegel} follows identically as the proof of \cite[Theorem~2.4]{Siegel-Drinfeld} once we replace \cite[Proposition~3.12]{Siegel-Drinfeld} by Proposition~\ref{prop:main}.

\def\cprime{$'$} \def\cprime{$'$} \def\cprime{$'$} \def\cprime{$'$}
\providecommand{\bysame}{\leavevmode\hbox to3em{\hrulefill}\thinspace}
\providecommand{\MR}{\relax\ifhmode\unskip\space\fi MR }
\providecommand{\MRhref}[2]{%
  \href{http://www.ams.org/mathscinet-getitem?mr=#1}{#2}
}
\providecommand{\href}[2]{#2}


\begin{thebibliography}{Den92b}


\bibitem[BGT06]{DML-book}
J.~P.~Bell, D.~Ghioca and T.~J.~Tucker, \emph{The Dynamical Mordell-Lang Conjecture}, Mathematical Surveys and Monographs, 2016, vol.~\textbf{210}, American Mathematical Society, Providence, RI, xiv+280 pp.



\bibitem[Bre05]{Breuer}
F.~Breuer, \emph{The {A}ndr\'{e}-{O}ort conjecture for products of {D}rinfeld
  modular curves}, J. reine angew. Math. \textbf{579} (2005), 115--144.


\bibitem[Den92a]{Denis-conjectures}
L.~Denis, \emph{G\'eom\'etrie diophantienne sur les modules de {D}rinfel\cprime
  d}, The arithmetic of function fields (Columbus, OH, 1991), Ohio State Univ.
  Math. Res. Inst. Publ., vol.~2, de Gruyter, Berlin, 1992, pp.~285--302.

\bibitem[Den92b]{Denis}
L.~Denis, \emph{Hauteurs canoniques et modules de {D}rinfel\cprime d}, Math.
  Ann. \textbf{294} (1992), no.~2, 213--223.

\bibitem[EY03]{Edixhoven}
B.~Edixhoven and A.~Yafaev, \emph{Subvarieties of {S}himura type}, Ann. of
  Math. (2) \textbf{157} (2003), no.~2, 621--645.

\bibitem[Ghi05]{IMRN}
D.~Ghioca, \emph{The {M}ordell-{L}ang theorem for {D}rinfeld modules}, Int.
  Math. Res. Not. (2005), no.~53, 3273--3307.

\bibitem[Ghi06]{Mat.Ann}
D.~Ghioca, \emph{Equidistribution for torsion points of a {D}rinfeld module},
  Math. Ann. \textbf{336} (2006), no.~4, 841--865.



\bibitem[Ghi07a]{mw2}
D.~Ghioca, \emph{The {L}ehmer inequality and the {M}ordell-{W}eil theorem for
  {D}rinfeld modules}, J. Number Theory \textbf{122} (2007), no.~1, 37--68.

\bibitem[Ghi07b]{locleh}
D.~Ghioca, \emph{The local {L}ehmer inequality for {D}rinfeld modules}, J. Number
  Theory \textbf{123} (2007), no.~2, 426--455.

\bibitem[Gos96]{Goss}
D.~Goss, \emph{Basic structures of function field arithmetic}, Ergebnisse der
  Mathematik und ihrer Grenzgebiete (3) [Results in Mathematics and Related
  Areas (3)], vol.~35, Springer-Verlag, Berlin, 1996.

\bibitem[GT07]{Siegel-Drinfeld}
D.~Ghioca and T.~J. Tucker, \emph{Siegel's theorem for Drinfeld modules}, Math.  Ann. \textbf{339} (2007), no.~1, 37--60.

\bibitem[GT08a]{findrin}
D.~Ghioca and T.~J. Tucker, \emph{Equidistribution and integral points for
  {D}rinfeld modules}, Trans. Amer. Math. Soc. \textbf{360} (2008), no.~7, 3839--3856.

\bibitem[GT08b]{dynml}
D.~Ghioca and T.~J. Tucker, \emph{A dynamical version of the {M}ordell-{L}ang conjecture for the additive group}, Compositio Math. \textbf{144} (2008), no.~2, 304--316. 

\bibitem[GT09]{GT-DML}
D.~Ghioca and T.~J.~Tucker, \emph{Periodic points, linearizing maps, and the dynamical Mordell-Lang problem}, J. Number Theory \textbf{129} (2009), no.~6, 1392--1403. 

\bibitem[Sca02]{Scanlon}
T.~Scanlon, \emph{Diophantine geometry of the torsion of a {D}rinfeld module},
  J. Number Theory \textbf{97} (2002), no.~1, 10--25.

\bibitem[Ser97]{Serre-Mordell_Weil}
J.-P. Serre, \emph{Lectures on the {M}ordell-{W}eil theorem}, third ed.,
  Aspects of Mathematics, Friedr. Vieweg \& Sohn, Braunschweig, 1997,
  Translated from the French and edited by Martin Brown from notes by Michel
  Waldschmidt, With a foreword by Brown and Serre.

\bibitem[Sie29]{Siegel}
C.~L. Siegel, \emph{{\"{U}}ber einige Anwendungen diophantischer
  Approximationen}, Abh. Preuss. Akad. Wiss. Phys. Math. Kl. (1929), 41--69.


\bibitem[SUZ97]{suz}
L.~Szpiro, E.~Ullmo, and S.~Zhang, \emph{Equir\'epartition des petits points},
  Invent.~Math. \textbf{127} (1997), 337--347.




\end{thebibliography}
\end{document}